\documentclass[11pt]{amsart}

\usepackage{geometry,verbatim,graphicx}
\usepackage[colorlinks=true,urlcolor=blue,
citecolor=red,linkcolor=blue,linktocpage,pdfpagelabels,bookmarksnumbered,bookmarksopen]{hyperref}

\newtheorem{lm}{Lemma}[section]
\newtheorem{prop}[lm]{Proposition}
\newtheorem{coro}[lm]{Corollary}
\newtheorem{teo}[lm]{Theorem}
\theoremstyle{definition}
\newtheorem{oss}[lm]{Remark}
\newtheorem*{ack}{Acknowledgments}

\title[An example]{A pathological example in\\ nonlinear spectral theory}

\author[Brasco]{Lorenzo Brasco}
\author[Franzina]{Giovanni Franzina}

\address[L.\ Brasco]{Dipartimento di Matematica e Informatica
	\newline\indent
	Universit\`a degli Studi di Ferrara
	\newline\indent
	Via Machiavelli 35, 44121 Ferrara, Italy}
\address{{\it and }
	Aix Marseille Univ, CNRS, Centrale Marseille, I2M, Marseille, France
	\newline\indent
	39 Rue Fr\'ed\'eric Joliot Curie, 13453 Marseille}
\email{lorenzo.brasco@unife.it}

\address[G.\ Franzina]{SISSA
\newline\indent
Via Bonomea 265, 34136 Trieste, Italy}
\email{gfranzina@sissa.it}


\numberwithin{equation}{section}

\begin{document}

\begin{abstract}
We construct an open set $\Omega\subset\mathbb{R}^N$ on which an eigenvalue problem for the $p-$Laplacian has not isolated first eigenvalue and the spectrum is not discrete. The same example shows that the usual Lusternik-Schnirelmann minimax construction does not exhaust the whole spectrum of this eigenvalue problem.
\end{abstract}

\maketitle

\begin{center}
\begin{minipage}{10cm}
\small
\tableofcontents
\end{minipage}
\end{center}

\section{Introduction}

\subsection{Framework}
For an open set $\Omega\subset\mathbb{R}^N$, we pick an exponent $1<p<\infty$ and consider the $p-$Laplace operator 
\[
\Delta_p u=\mathrm{div}(|\nabla u|^{p-2}\,\nabla u),
\]
acting on the homogeneous Sobolev space $\mathcal{D}^{1,p}_0(\Omega)$. The latter is defined as the completion of $C^\infty_0(\Omega)$ with respect to the norm
\[
u\mapsto \left(\int_\Omega |\nabla u|^p\,dx\right)^\frac{1}{p},\qquad \mbox{ for } u\in C^\infty_0(\Omega).
\]
The usual eigenvalue problem for the $p-$Laplace operator with homogeneous Dirichlet boundary condition is the following: find the numbers $\lambda\in\mathbb{R}$ such that the boundary value problem
\begin{equation}
\label{autovalore}
-\Delta_p u=\lambda\,|u|^{p-2}\,u,\ \mbox{ in } \Omega,\qquad u=0, \ \mbox{ on }\partial\Omega,
\end{equation}
admits a solution $u\in\mathcal{D}^{1,p}_0(\Omega)\setminus\{0\}$, see for example \cite{Lqv}.
\vskip.2cm\noindent
In this note we want to consider the following variant
\begin{equation}
\label{autosalone}
-\Delta_p u=\lambda\,\|u\|^{p-q}_{L^q(\Omega)}\, |u|^{q-2}\,u,\ \mbox{ in } \Omega,\qquad u=0, \ \mbox{ on }\partial\Omega,
\end{equation}
where $1<q<p$. This problem has already been studied by the second author and Lamberti in \cite{FL}. At a first glance, equation \eqref{autosalone} could seem
a bit
weird, due to the presence of the $L^q$ norm on the right-hand side. We observe that this term guarantees that both sides of the equation share the same homogeneity, exactly like in the standard case \eqref{autovalore}. 
\par
Though the introduction of this term containing the $L^q$ norm may looks artificial, nevertheless it is easily seen that \eqref{autosalone} is a natural extension of \eqref{autovalore}. Indeed, eigenvalues of the $p-$Laplacian can be seen as critical points of the functional $u\mapsto \int_\Omega |\nabla u|^p\,dx$ restricted to the manifold
\[
\mathcal{S}_p(\Omega)=\{u\in \mathcal{D}^{1,p}_0(\Omega)\, :\, \|u\|_{L^p(\Omega)}=1\}.
\]
In a similar fashion, eigenvalues of \eqref{autosalone} correspond to critical points of the same  functional, this time restricted to the manifold
\[
\mathcal{S}_{p,q}(\Omega)=\{u\in \mathcal{D}^{1,p}_0(\Omega)\, :\, \|u\|_{L^q(\Omega)}=1\}.
\]
We define the {\it $(p,q)-$spectrum of $\Omega$} as follows
\[
\mathrm{Spec}(\Omega;p,q)=\{\lambda\in \mathbb{R}\, :\, \mbox{equation \eqref{autosalone} admits a solution in } \mathcal{D}^{1,p}_0(\Omega)\setminus\{0\}\},
\]
and we
call every element of this set a {\it $(p,q)-$eigenvalue of $\Omega$}.
\par 
Let us assume that the open set $\Omega\subset\mathbb{R}^N$ is such that the embedding $\mathcal{D}^{1,p}_0(\Omega)\hookrightarrow L^q(\Omega)$ is compact.
It is known that $\mathrm{Spec}(\Omega;p,q)$ is a closed set, see \cite[Theorem 5.1]{FL}. It is not difficult to see that 
\[
\lambda\ge \lambda_{p,q}^1(\Omega)>0,\qquad \mbox{ for every }\lambda \in\mathrm{Spec}(\Omega;p,q),
\]
where $\lambda_{p,q}^1(\Omega)$ is the {\it first $(p,q)-$eigenvalue of $\Omega$}, defined by
\[
\lambda_{p,q}^1(\Omega)=\min_{u\in \mathcal{S}_{p,q}(\Omega)} \int_\Omega |\nabla u|^p\,dx. 
\] 
We recall that when $\Omega$ is connected, then $\lambda_{p,q}^1(\Omega)$ is {\it simple}, i.e. the corresponding solutions to \eqref{autosalone} forms a vector subspace of dimension $1$ (see \cite[Theorem 3.1]{FL}).
\par
Moreover, it is known that $\mathrm{Spec}(\Omega;p,q)$ contains an increasingly diverging sequence of eigenvalues $\{\lambda_{p,q}^k(\Omega)\}_{k\in\mathbb{N}\setminus\{0\}}$, defined through a variational procedure analogous to the so-called {\it Courant minimax principle} used for the spectrum of the Laplacian.
\par
Let us be more precise on this point. For every $k\in\mathbb{N}\setminus\{0\}$, we define
\[
\Sigma^k_{p,q}(\Omega)=\Big\{A\subset\mathcal{S}_{p,q}(\Omega)\, :\, A \mbox{ compact and symmetric, with } \gamma(A)\ge k \Big\},
\]
where $\gamma(\cdot)$ denotes the {\it Krasnosel'ski\u{\i} genus} of a closed set, defined by
\[
\gamma(A)=\inf\left\{k\in\mathbb{N}\, :\, \exists \mbox{ a continuous odd map } \phi:A\to\mathbb{S}^{k-1}\right\},
\]
with the convention that $\gamma(A)=+\infty$, if no such an integer $k$ exists. Then for every $k\in\mathbb{N}\setminus\{0\}$, one can define the number
\[
\lambda^k_{p,q}(\Omega)=\inf_{A\in\Sigma^k_{p,q}(\Omega)}\max_{u\in A} \int_\Omega |\nabla u|^p\, dx.
\]
By \cite[Theorem 5.2]{FL} we have
\[
\{\lambda^k_{p,q}(\Omega)\}_{k\in\mathbb{N}\setminus\{0\}}\subset \mathrm{Spec}(\Omega;p,q)\qquad \mbox{ and }\qquad \lim_{k\to\infty} \lambda^k_{p,q}(\Omega)=+\infty.
\]
We will use the notation
\[
\mathrm{Spec}_{LS}(\Omega;p,q):=\{\lambda^k_{p,q}(\Omega)\}_{k\in\mathbb{N}\setminus\{0\}},
\]
for the {\it Lusternik-Schnirelmann $(p,q)-$spectrum of $\Omega$}.
\par
We recall that when $p=q=2$ then the Lusternik-Schnirelmann spectrum coincides with the whole spectrum of the Dirichlet-Laplacian, see for example \cite[Theorem A.2]{BPS}. In all the other cases, it is not known whether $\mathrm{Spec}_{LS}(\Omega;p,q)$ and $\mathrm{Spec}(\Omega;p,q)$ coincide or not.

\subsection{The content of the paper} The humble aim of this small note is to shed some light on the relation between the two spectra. More precisely, in Theorem \ref{teo:palle} below
we construct an example of an open set $\mathcal{B}\subset\mathbb{R}^N$ such that for $1<q<p$
\begin{itemize}
\item the embedding $\mathcal{D}^{1,p}_0(\mathcal{B})\hookrightarrow L^q(\mathcal{B})$ is compact (the set $\mathcal{B}$ is indeed bounded);
\vskip.2cm
\item $\mathrm{Spec}_{LS}(\mathcal{B};p,q)\not =\mathrm{Spec}(\mathcal{B};p,q)$;
\vskip.2cm
\item $\mathrm{Spec}(\mathcal{B};p,q)$ has (at least) countably many accumulation points.
\end{itemize}
Actually, by using the same idea, in Theorem \ref{teo:grovierone} below we present an even worse example, i.e. an open set $\mathcal{T}\subset\mathbb{R}^N$ such that for $1<q<p$
\begin{itemize}
\item the embedding $\mathcal{D}^{1,p}_0(\mathcal{T})\hookrightarrow L^q(\mathcal{T})$ is compact;
\vskip.2cm
\item $\mathrm{Spec}_{LS}(\mathcal{T};p,q)\not =\mathrm{Spec}(\mathcal{T};p,q)$.
\vskip.2cm
\item $\mathrm{Spec}(\mathcal{T};p,q)$ has (at least) countably many accumulation points;
\vskip.2cm
\item the first eigenvalue $\lambda^1_{p,q}(\mathcal{T})$ is {\it not} isolated, i.e. there exists $\{\lambda_n\}_n\subset \mathrm{Spec}(\mathcal{T};p,q)$ such that 
\[
\lambda^1_{p,q}(\mathcal{T})=\lim_{n\to\infty} \lambda_n.
\]
\end{itemize}
Although we agree that our examples are quite pathological (in particular $\mathcal{T}$ could be bounded, but made of infinitely many connected components) and strongly based on the fact that $q/p<1$, we believe them to have their own interest in abstract Critical Point Theory.  
\begin{oss}[More general index theories]
For simplicity, in this paper we consider the Lusternik-Schnirelmann spectrum defined by means of the Krasnosel'ski\u{\i} genus. We recall that it is possible to define diverging sequences of eigenvalues in a similar fashion, by using another index in place of the genus. For example, one could use the {\it $\mathbb{Z}_2-$cohomological index}
\cite{fadell_rabinowitz1977} or the {\it Lusternik-Schnirelmann Category} \cite[Chapter 2]{St}. Our examples still apply in each of these cases, since they are independent of the choice of the index.
\end{oss}

\begin{ack}
We thank Peter Lindqvist for his kind interest in this work. This manuscript has been finalized while the first author was visiting the KTH (Stockholm) in February 2017. He wishes to thank Erik Lindgren for the kind invitation. The authors are members of the Gruppo Nazionale per l'Analisi Matematica, la Pro\-ba\-bi\-li\-t\`a
e le loro Applicazioni (GNAMPA) of the Istituto Nazionale di Alta Matematica (INdAM).
\end{ack}

\section{Spectrum of disconnected sets}
\subsection{General eigenvalues}

For the standard eigenvalue problem \eqref{autovalore}, i.e. when $q=p$, it is well-known that the spectrum of a disconnected open set $\Omega$ is made of the collection of the eigenvalues of its connected components. For $1<q<p$ this only gives a part of the spectrum, the general formula is contained in the following result.
\begin{prop}
\label{prop:generalone}
Let $1<q<p<\infty$ and  let $\Omega\subset\mathbb{R}^N$ be an open set such that $\mathcal{D}^{1,p}_0(\Omega)\hookrightarrow L^q(\Omega)$ is compact.  Let us suppose that
\[
\Omega=\Omega_1\cup \Omega_2,
\]
with $\Omega_i\subset \mathbb{R}^N$ open set, such that $\mathrm{dist}(\Omega_1,\Omega_2)>0$.
Then $\lambda$ is a $(p,q)-$eigenvalue of $\Omega$ if and only if it is of the form
\begin{equation}
\label{rap}
\lambda=\left[\displaystyle\left(\frac{\delta_1}{\lambda_1}\right)^\frac{q}{p-q}+\displaystyle\left(\frac{\delta_2}{\lambda_2}\right)^\frac{q}{p-q}\right]^\frac{q-p}{q}\qquad \mbox{ for some $(p,q)-$eigenvalue $\lambda_i$ of $\Omega_i$},
\end{equation}
where the coefficients $\delta_1$ and $\delta_2$ are such that
\[
\delta_i\in\{0,1\}\qquad \mbox{ and }\qquad \delta_1+\delta_2\not =0.
\]
Moreover, if we set 
\[
|\alpha_i|=\displaystyle\left(\frac{\lambda}{\lambda_i}\right)^\frac{1}{p-q},\qquad i=1,2,
\]
each $(p,q)-$eigenfunction $U$ of $\Omega$ corresponding to \eqref{rap} takes the form
\begin{equation}
\label{autofunzioni}
U=C\,\Big(\delta_1\,\alpha_1\,u_1+\delta_2\,\alpha_2\,u_2\Big),
\end{equation}
where $C\in\mathbb{R}$ and $u_i\in\mathcal{D}^{1,p}_0(\Omega_i)$ is a $(p,q)-$eigenfunction of $\Omega_i$ with unitary $L^q$ norm corresponding to $\lambda_i$, for $i=1,2$.
\end{prop}
\begin{proof}
Let us suppose that $\lambda$ is an eigenvalue and let $U\in \mathcal{D}^{1,p}_0(\Omega)$ be a corresponding eigenfunction. For simplicity, we take $U$ with unitary $L^q$ norm. Let us set
\[
u_i=U\cdot 1_{\Omega_i}\in \mathcal{D}^{1,p}_0(\Omega_i),\qquad i=1,2,
\]
then these two functions are weak solutions of
\[
-\Delta_p u_i=\lambda\, |u_i|^{q-2}\, u_i,\quad \mbox{ in }\Omega_i,\ i=1,2.
\]
We have to distinguish two situations: either both $u_1$ and $u_2$ are not identically zero; or at least one of the two identically vanishes.
\par
In the first case, by setting $\alpha_i=\|u_i\|_{L^q(\Omega_i)}$, for $i=1,2$, we can rewrite the previous equation as
\[
-\Delta_p u_i=\frac{\lambda}{\alpha_i^{p-q}}\,\|u_i\|^{p-q}_{L^q(\Omega_i)}\, |u_i|^{q-2}\, u_i,\quad \mbox{ in }\Omega_i,\ i=1,2,
\]
which implies that $\lambda_i:=\lambda\, \alpha_i^{q-p}$ is an eigenvalue of $\Omega_i$, $i=1,2$. By using that $\alpha_1^q+\alpha_2^q=1$, we can infer that
\[
1=\alpha_1^q+\alpha_2^q=\lambda^\frac{q}{p-q}\, \left[\displaystyle\left(\frac{1}{\lambda_1}\right)^\frac{q}{p-q}+\displaystyle\left(\frac{1}{\lambda_2}\right)^\frac{q}{p-q}\right],
\]
which implies that $\lambda$ has the form \eqref{rap}, with $\delta_1=\delta_2=1$. Moreover, since $\lambda\,\alpha_i^{q-p}=\lambda_i$, this gives that the eigenfunction $U$ has the form
\[
\begin{split}
U=u_1+u_2&=\alpha_1\,\frac{u_1}{\|u_1\|_{L^q(\Omega_1)}}+\alpha_2\,\frac{u_2}{\|u_2\|_{L^q(\Omega_2)}}\\
&=\left(\frac{\lambda}{\lambda_1}\right)^\frac{1}{p-q}\,\frac{u_1}{\|u_1\|_{L^q(\Omega_1)}}+\left(\frac{\lambda}{\lambda_2}\right)^\frac{1}{p-q}\,\frac{u_2}{\|u_2\|_{L^q(\Omega_2)}},
\end{split}
\]
which is formula \eqref{autofunzioni}.
\par
Let us now suppose that $u_2\equiv 0$, this implies that $U=u_1$ and $u_1$ has unitary $L^q$ norm. This automatically gives that $\lambda$ is an eigenvalue of $\Omega_1$, i.e. we have formula \eqref{rap} with $\delta_1=1$ and $\delta_2=0$.
\vskip.2cm\noindent
Conversely, let us now suppose that $\lambda_i$ is a $(p,q)-$eigenvalue of $\Omega_i$ with eigenfunction $u_i\in\mathcal{D}^{1,}_0(\Omega_i)$ normalized in $L^q$, for $i=1,2$. We are going to prove that formula \eqref{rap} gives a $(p,q)-$eigenvalue of $\Omega$.
\par
 We first observe that we immediately get that $\lambda_1$ and $\lambda_2$ are eigenvalues of $\Omega$, with eigenfunctions $u_1$ and $u_2$ extended by $0$ on the other component. This corresponds to \eqref{rap} with $\delta_2=0$ and $\delta_1=0$, respectively.
\par
Now we set
\[
U=\beta_1\, u_1+\beta_2\, u_2\in \mathcal{D}^{1,p}_0(\Omega), 
\] 
where $\beta_1,\beta_2\in\mathbb{R}\setminus\{0\}$ has to be suitably chosen. By using the equations solved by $u_1$ and $u_2$ and using that these have disjoint supports, we get that
\[
\begin{split}
-\Delta_p U=-|\beta_i|^{p-2}\,\beta_i\, \Delta_p u_i&=|\beta_i|^{p-2}\,\beta_i\, \lambda_i\, |u_i|^{q-2}\, u_i\\
&=|\beta_i|^{p-q}\, \lambda_i\,|U|^{q-2}\, U,\qquad \mbox{ in }\Omega_i,\quad i=1,2.
\end{split}
\]
The previous implies that if we want $U$ to be an eigenfunction of $\Omega$ with eigenvalue $\lambda$ given by formula \eqref{rap} with $\delta_1=\delta_2=1$, we need to choose $\beta_1,\beta_2$ in such a way that
\[
|\beta_1|^{p-q}\, \lambda_1=\lambda\, \|U\|^{p-q}_{L^q(\Omega)}=|\beta_2|^{p-q}\, \lambda_2.
\]
Since we have
\[
\|U\|^{p-q}_{L^q(\Omega)}=\left(|\beta_1|^q+|\beta_2|^q\right)^\frac{p-q}{q},
\]
this is equivalent to require that
\[
|\beta_1|^{p-q}\, \lambda_1=\left[\displaystyle\left(\frac{1}{\lambda_1}\right)^\frac{q}{p-q}+\displaystyle\left(\frac{1}{\lambda_2}\right)^\frac{q}{p-q}\right]^\frac{q-p}{q}\, \left(|\beta_1|^q+|\beta_2|^q\right)^\frac{p-q}{q},
\]
and
\[
|\beta_2|^{p-q}\, \lambda_2=\left[\displaystyle\left(\frac{1}{\lambda_1}\right)^\frac{q}{p-q}+\displaystyle\left(\frac{1}{\lambda_2}\right)^\frac{q}{p-q}\right]^\frac{q-p}{q}\, \left(|\beta_1|^q+|\beta_2|^q\right)^\frac{p-q}{q},
\]
that is
\[
|\beta_1|=\left(\frac{\lambda_2}{\lambda_1}\right)^\frac{1}{p-q}\,|\beta_2|.
\]
Thus we get that $U$ must be of the form \eqref{autofunzioni}, in the case $\delta_1=\delta_2=1$. Moreover, we obtain that formula \eqref{rap} with $\delta_1=\delta_2=1$ defines an eigenvalue of $\Omega$.
\end{proof}
We can iterate the previous result and get the following
\begin{coro}
\label{coro:crucial}
Let $1<q<p<\infty$ and let $\Omega\subset\mathbb{R}^N$ be an open set such that $\mathcal{D}^{1,p}_0(\Omega)\hookrightarrow L^q(\Omega)$ is compact.
Let us suppose that
\[
\Omega=\bigcup_{i=1}^\# \Omega_i,
\]
with $\Omega_i\subset \mathbb{R}^N$ open set, such that $\mathrm{dist}(\Omega_i,\Omega_j)>0$, for $i\not =j$. Then $\lambda$ is a $(p,q)-$eigenvalue of $\Omega$ if and only if it is of the form
\begin{equation}
\label{rap_gen}
\lambda=\left[\displaystyle\sum_{i=1}^\#\left(\frac{\delta_i}{\lambda_i}\right)^\frac{q}{p-q}\right]^\frac{q-p}{q}\qquad \mbox{ for some $(p,q)-$eigenvalue $\lambda_i$ of $\Omega_i$},
\end{equation}
where the coefficients $\delta_i$ are such that
\[
\delta_i\in\{0,1\}\qquad \mbox{ and }\qquad \sum_{i=1}^\#\delta_i\not =0.
\]
Moreover, if we set 
\[
|\alpha_i|=\displaystyle\left(\frac{\lambda}{\lambda_i}\right)^\frac{1}{p-q},
\]
each corresponding $(p,q)-$eigenfunction $U$ of $\Omega$ has the form
\[
U=C\,\left(\sum_{i=1}^\#\delta_i\,\alpha_i\,u_i\right),
\]
where $C\in\mathbb{R}$ and $u_i\in\mathcal{D}^{1,p}_0(\Omega)$ is $(p,q)-$eigenfunction of $\Omega_i$ with unitary $L^q$ norm corresponding to $\lambda_i$.
\end{coro}
\begin{oss}
When $\#=\infty$, i.e. $\Omega$ has infinitely many connected components, formula \eqref{rap_gen} above has to be interpreted in the usual sense
\[
\left[\displaystyle\sum_{i=1}^\#\left(\frac{\delta_i}{\lambda_i}\right)^\frac{q}{p-q}\right]^\frac{q-p}{q}=\lim_{k\to\infty} \left[\displaystyle\sum_{i=1}^k\left(\frac{\delta_i}{\lambda_i}\right)^\frac{q}{p-q}\right]^\frac{q-p}{q},
\]
since the limit exists by monotonicity. We also observe that since $q-p<0$, if $\delta_k=1$ then we have
\[
\left[\displaystyle\sum_{i=1}^\#\left(\frac{\delta_i}{\lambda_i}\right)^\frac{q}{p-q}\right]^\frac{q-p}{q}\le \lambda_k<+\infty.
\]
On the other hand, since for every $k\in\mathbb{N}$, the formula 
\[
\left[\displaystyle\sum_{i=1}^k\left(\frac{\delta_i}{\lambda_i}\right)^\frac{q}{p-q}\right]^\frac{q-p}{q}
\]
gives a $(p,q)-$eigenvalue of $\Omega$, by recalling that $\lambda^1_{p,q}(\Omega)$ is the least eigenvalue we obtain
\[
\left[\displaystyle\sum_{i=1}^\#\left(\frac{\delta_i}{\lambda_i}\right)^\frac{q}{p-q}\right]^\frac{q-p}{q}=\lim_{k\to\infty} \left[\displaystyle\sum_{i=1}^k\left(\frac{\delta_i}{\lambda_i}\right)^\frac{q}{p-q}\right]^\frac{q-p}{q}\ge \lambda^1_{p,q}(\Omega)>0.
\]
\end{oss}

\subsection{The first eigenvalue}

Thanks to the formula of Proposition \ref{prop:generalone}, we can now compute the first $(p,q)-$eigenvalue of a disconnected set. For ease of readability, we start as before with the case of two connected components.
\begin{coro}
\label{prop:primo}
Let $1<q<p<\infty$ and let $\Omega\subset\mathbb{R}^N$ be an open set such that $\mathcal{D}^{1,p}_0(\Omega)\hookrightarrow L^q(\Omega)$ is compact. Let us suppose that
\[
\Omega=\Omega_1\cup \Omega_2,
\]
with $\Omega_i\subset \mathbb{R}^N$ open connected set, such that $\mathrm{dist}(\Omega_1,\Omega_2)>0$.
Then we have
\begin{equation}
\label{sub}
\lambda^1_{p,q}(\Omega)=\left[\left(\frac{1}{\lambda^1_{p,q}(\Omega_1)}\right)^\frac{q}{p-q}+\left(\frac{1}{\lambda^1_{p,q}(\Omega_2)}\right)^\frac{q}{p-q}\right]^\frac{q-p}{q}.
\end{equation}
Moreover, each first $(p,q)-$eigenfunction of $\Omega$ with unitary $L^q$ norm has the form
\begin{equation}
\label{autofunzioni1}
\alpha_1\,u_1+\alpha_2\,u_2,\qquad \mbox{ where }|\alpha_i|=\left(\frac{\lambda^1_{p,q}(\Omega)}{\lambda_{p,q}^1(\Omega_i)}\right)^\frac{1}{p-q},
\end{equation}
and $u_i\in\mathcal{D}^{1,p}_0(\Omega)$ is the first positive $(p,q)-$eigenfunction of $\Omega_i$ with unitary $L^q$ norm, for $i=1,2$.
\end{coro}
\begin{proof}
From formula \eqref{rap}, we already know that we must have
\begin{equation}
\label{minimizza}
\lambda^1_{p,q}(\Omega)=\left[\displaystyle\left(\frac{\delta_1}{\lambda_1}\right)^\frac{q}{p-q}+\displaystyle\left(\frac{\delta_2}{\lambda_2}\right)^\frac{q}{p-q}\right]^\frac{q-p}{q}\qquad \mbox{ for some eigenvalue $\lambda_i$ of $\Omega_i$}.
\end{equation}
We now observe that the function
\[
\Phi(s,t)=\left[s^\frac{q}{p-q}+t^\frac{q}{p-q}\right]^\frac{q-p}{q},\qquad (s,t)\in \Big([0,+\infty)\times [0,+\infty)\Big)\setminus\{(0,0)\},
\]
is decreasing in both variables (here we use that $q<p$). This implies that the right-hand side of \eqref{minimizza} is minimal when
\[
\delta_1=\delta_2=1,\qquad \lambda_1=\lambda^1_{p,q}(\Omega_1)\qquad \mbox{ and }\qquad \lambda_2=\lambda^1_{p,q}(\Omega_2),
\]
i.e. formula \eqref{sub}. The representation formula \eqref{autofunzioni1} now follows from that of Proposition \ref{prop:generalone}.
\end{proof}
\begin{oss}
Under the assumptions of the previous result, we obtain in particular that $\Omega=\Omega_1\cup \Omega_2$ has exactly $4$ first $(p,q)-$eigenfunctions with unitary $L^q$ norm, given by
\[
|\alpha_1|\,u_1+|\alpha_2|\,u_2,\qquad |\alpha_1|\,u_1-|\alpha_2|\,u_2,\qquad -|\alpha_1|\,u_1+|\alpha_2|\,u_2\qquad \mbox{ and }\qquad -|\alpha_1|\,u_1-|\alpha_2|\,u_2.
\]
In particular, although $\lambda^1_{p,q}(\Omega)$ is not simple in this situation, however the collection of the first eigenfunctions on $\mathcal{S}_{p,q}(\Omega)$ is a set of genus $1$. 
\par
This phenomenon disappears in the limit case $p=q$, if the two components $\Omega_1$ and $\Omega_2$ have the same first eigenvalue: indeed, in this case the first eigenfunctions on $\mathcal{S}_{p}(\Omega)$ forms the set of genus $2$
\[
\{\alpha\,u_1+\beta\,u_2\, :\, |\alpha|^p+|\beta|^p=1\}.
\]
\end{oss}
More generally, we get the following
\begin{coro}
\label{coro:first}
Let $1<q<p<\infty$ and let $\Omega\subset\mathbb{R}^N$ be an open set such that $\mathcal{D}^{1,p}_0(\Omega)\hookrightarrow L^q(\Omega)$ is compact.
Let us suppose that
\[
\Omega=\bigcup_{i=1}^\# \Omega_i,
\]
with $\Omega_i\subset \mathbb{R}^N$ open set, such that $\mathrm{dist}(\Omega_i,\Omega_j)>0$, for $i\not =j$. Then we have
\[
\lambda^1_{p,q}(\Omega)=\left[\displaystyle\sum_{i=1}^\#\left(\frac{1}{\lambda^1_{p,q}(\Omega_i)}\right)^\frac{q}{p-q}\right]^\frac{q-p}{q}.
\]
Moreover, each corresponding first $(p,q)-$eigenfunction of $\Omega$ with unitary $L^q$ norm has the form
\[
\sum_{i=1}^\#\alpha_i\,u_i,\qquad \mbox{ where }|\alpha_i|=\displaystyle\left(\frac{\lambda^1_{p,q}(\Omega)}{\lambda_{p,q}^1(\Omega_i)}\right)^\frac{1}{p-q},
\]
and $u_i\in\mathcal{D}^{1,p}_0(\Omega)$ is a first $(p,q)-$eigenfunction of $\Omega_i$ with unitary $L^q$ norm corresponding to $\lambda_i$.
\end{coro}

\section{Construction of the examples}
We are now ready for the main results of this note. 
\begin{teo}
\label{teo:palle}
Let $1<q<p<\infty$ and $0<r\le R$, we take the disjoint union of balls
\[
\mathcal{B}=B_R(x_0)\cup B_r(y_0),\qquad \mbox{ with } |x_0-y_0|>R+r.
\] 
Then
\begin{equation}
\label{diversi!}
\mathrm{Spec}_{LS}(\mathcal{B};p,q)\not =\mathrm{Spec}(\mathcal{B};p,q).
\end{equation}
Moreover, the set $\mathrm{Spec}(\mathcal{B};p,q)$ has (at least) countably many accumulation points.
\end{teo}
\begin{proof}
We observe that for every $k\ge 1$ there exists a sequence $\{\lambda_{n,k}\}_{n\in\mathbb{N}}\subset\mathrm{Spec}(\mathcal{B};p,q)$ such that 
\begin{equation}
\label{saltelli}
\lambda^k_{p,q}(B_R(x_0))=\lim_{n\to\infty} \lambda_{n,k}.
\end{equation}
Namely,
\[
\lambda_{n,k}=\left[\displaystyle\left(\frac{1}{\lambda^k_{p,q}(B_R(x_0))}\right)^\frac{q}{p-q}+\displaystyle\left(\frac{1}{\lambda^n_{p,q}(B_r(y_0))}\right)^\frac{q}{p-q}\right]^\frac{q-p}{q},
\]
is a $(p,q)-$eigenvalue of $\mathcal{B}$ for all $n\ge1$, thanks to formula \eqref{rap}, and we have that
\[
\lim_{n\to\infty}\lambda^n_{p,q}(B_r(y_0))=+\infty.
\]
From \eqref{saltelli} we immediately deduce the second part of the statement, since $\lambda^k_{p,q}(B_R(x_0))$
belongs to $\mathrm{Spec}(\mathcal{B};p,q)$ by formula \eqref{rap}. Moreover, \eqref{saltelli} implies \eqref{diversi!} as well.
Indeed, if the two spectra were the same then
\[
\mathrm{Spec}(\mathcal{B};p,q)=\{\lambda^k_{p,q}(\mathcal{B})\}_{k\in\mathbb{N}\setminus\{0\}}
\]
would be an increasing sequence diverging to $+\infty$, with (infinitely many) accumulation points,
which is impossibile.
\end{proof}
We can refine the previous construction and obtain that for our eigenvalue problem even the isolation of the first eigenvalue may fail, in general.
\begin{teo}
\label{teo:grovierone}
Let $1<q<p$ and let $\{r_i\}_{i\in\mathbb{N}}\subset \mathbb{R}$ be a sequence of strictly positive numbers, such that
\begin{equation}
\label{compatto}
\sum_{i=0}^\infty r_i^{\frac{p\,q}{p-q}+N}<+\infty.
\end{equation}
We then define the sequence of points $\{x_i\}_{i\in\mathbb{N}}\subset\mathbb{R}^N$ by
\[
\left\{\begin{array}{rcl}
x_0&=&(0,\dots,0),\\
x_{i+1}&=&(2^{-i}+r_i+r_{i+1},0,\dots,0)+x_i,
\end{array}
\right.
\]
and the disjoint union of balls 
\[
\mathcal{T}=\bigcup_{i=0}^\infty B_{r_i}(x_i).
\]
Then 
\[
\mathrm{Spec}_{LS}(\mathcal{T};p,q)\not =\mathrm{Spec}(\mathcal{T};p,q).
\]
and the set $\mathrm{Spec}(\mathcal{T};p,q)$ has (at least) countably many accumulation points. Moreover, the first eigenvalue $\lambda^1_{p,q}(\mathcal{T})$ is not isolated. 
\end{teo}
\begin{proof}
We first observe that the condition \eqref{compatto} guarantess the compactness of $\mathcal{D}^{1,p}_0(\mathcal{T})\hookrightarrow L^q(\mathcal{T})$, see \cite[Theorem 1.2 \& Example 5.2]{braruf}. The first statement follows as in the previous theorem.
\par
In order to prove that $\lambda^1_{p,q}(\mathcal{T})$ is an accumulation point of the spectrum, we can now use Corollaries \ref{coro:first} and \ref{coro:crucial} to construct a sequence of eigenvalues $\{\lambda_n\}_{n\in\mathbb{N}}$ such that $\lambda_n$ converges to $\lambda^1_{p,q}(\mathcal{T})$. We just set
\[
\lambda_n=\left[\displaystyle\sum_{i=1}^n\left(\frac{1}{\lambda^1_{p,q}(B_{r_i}(x_i))}\right)^\frac{q}{p-q}\right]^\frac{q-p}{q}.
\]
This gives the desired sequence. 
\end{proof}
\begin{figure}[h]
\includegraphics[scale=.25]{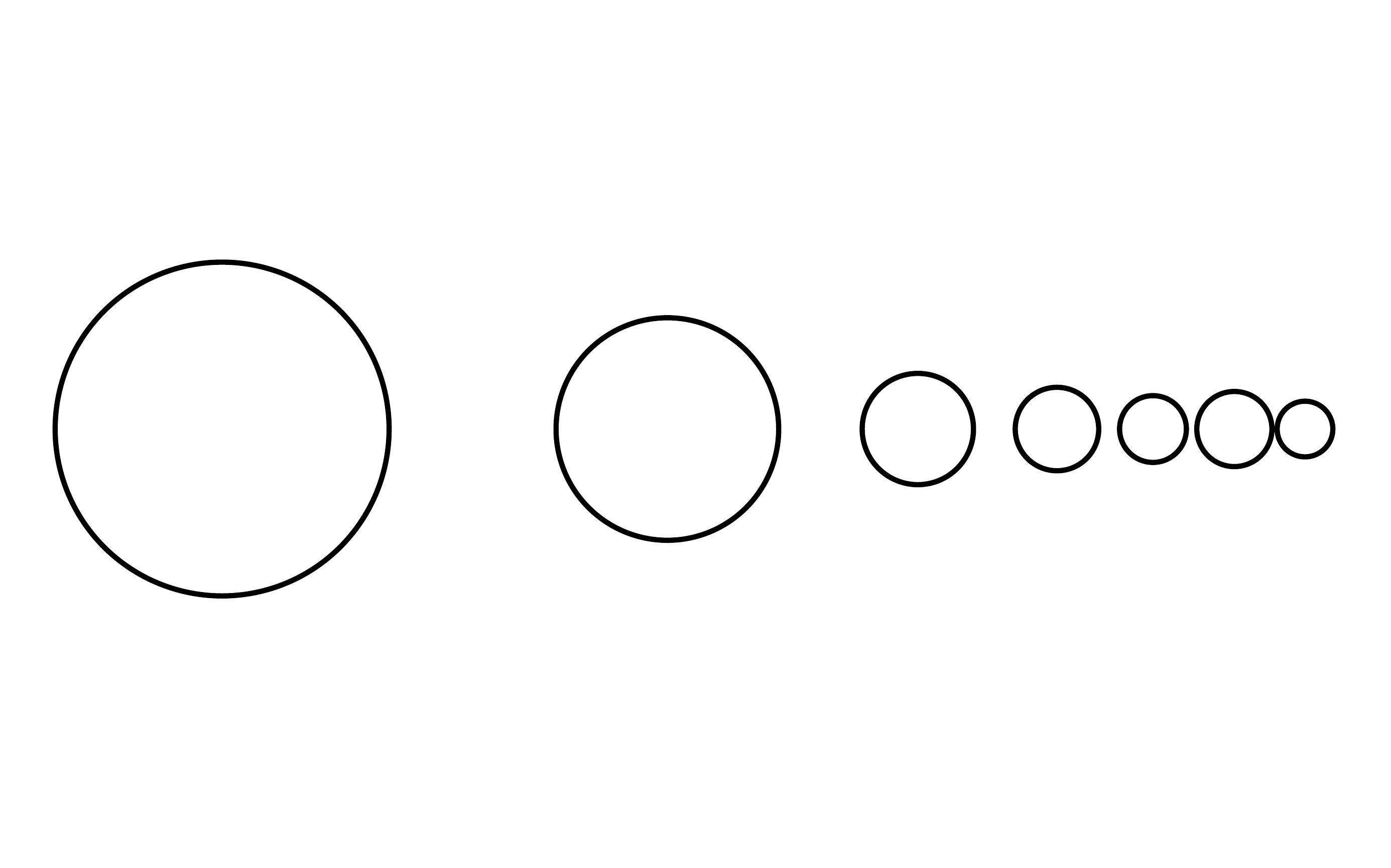}
\caption{The set $\mathcal{T}$ is a disjoint union of countably many shrinking balls. }
\end{figure}
\begin{oss}
The examples above are given in terms of disjoint unions of balls just for simplicity. Actually, they still work with disjoint unions of more general bounded sets.
\end{oss}

\end{document}